\documentclass[a4paper]{amsart}

\pdfoutput=1

\usepackage[utf8x]{inputenc}
\usepackage{amsmath, amsfonts, amssymb, amscd, constants}
\usepackage[pdftex]{graphicx}
\usepackage{color}
\usepackage{graphicx}
\DeclareGraphicsExtensions{.png,.pdf}

\usepackage[a4paper,text={14cm,21cm},centering]{geometry}

\usepackage[small,bf]{caption}
\newtheorem{theorem}{Theorem}[section]
\newtheorem{corollary}{Corollary}[section]

\newtheorem{lemma}{Lemma}[section]
\newtheorem{proposition}{Proposition}[section]

\theoremstyle{definition}
\newtheorem{remark}{Remark}[section]

\numberwithin{equation}{section}

\newcommand{\calC}{\mathcal{C}}
\newcommand{\calD}{\mathcal{D}}

\newcommand{\calF}{\mathcal{F}}
\newcommand{\calG}{\mathcal{G}}

\newcommand{\calI}{\mathcal{I}}

\newcommand{\calL}{\mathcal{L}}

\newcommand{\calS}{\mathcal{S}}

\newcommand{\frF}{\mathfrak{F}}

\newcommand{\bbR}{\mathbb{R}}
\newcommand{\bbS}{\mathbb{S}}

\newcommand{\bbZ}{\mathbb{Z}}

\newcommand{\bfZ}{\mathbf{Z}}

\newcommand{\bfb}{\mathbf{b}}

\newcommand{\comp}{{\rm c}}
\newcommand{\defby}{\doteqdot}
\newcommand{\normsup}[1]{\|#1\|_{\scriptscriptstyle\infty}}
\newcommand{\normI}[1]{\|#1\|_{\scriptscriptstyle 1}}
\newcommand{\normII}[1]{\|#1\|_{\scriptscriptstyle 2}}
\newcommand{\setof}[2]{\{#1\,:\,#2\}}

\newcommand{\inprod}[2]{(#1,#2)}
\newcommand{\bk}[1]{\langle#1\rangle}
\newcommand{\betac}{\beta_{\rm\scriptscriptstyle c}}
\newcommand{\ub}{\bfb}
\newcommand{\ubo}{\ub(\omega)}
\newcommand{\uG}{\mathbf{\Gamma}}
\newcommand{\GM}{\mu^\omega_{\Lambda_n;\beta}}
\newcommand{\minusconn}{\stackrel{-}{\longleftrightarrow}}

\title{A Finite-Volume Version of Aizenman-Higuchi Theorem for the 2d Ising Model}
\author{Loren Coquille \and Yvan Velenik}
\address{Department of Mathematics, University of Geneva, 2-4 rue du Lièvre, Case Postale 64, CH-1211 Genève 4. }
\email{Loren.Coquille@unige.ch, Yvan.Velenik@unige.ch}

\begin{document}

\begin{abstract}
In the late 1970s, in two celebrated papers, Aizenman and Higuchi independently established that all infinite-volume Gibbs measures of the two-dimensional ferromagnetic nearest-neighbor Ising model at inverse temperature $\beta\geq 0$ are of the form $\alpha\mu^+_\beta + (1-\alpha)\mu^-_\beta$, where $\mu^+_\beta$ and $\mu^-_\beta$ are the two pure phases and $0\leq\alpha\leq 1$. We present here a new approach to this result, with a number of advantages: (i) We obtain an optimal finite-volume, quantitative analogue (implying the classical claim); (ii) the scheme of our proof seems more natural and provides a better picture of the underlying phenomenon; (iii) this new approach might be applicable to systems for which the classical method fails.
\end{abstract}

\keywords{Ising model -- Gibbs states -- translation invariance}

\maketitle

\section{Introduction and results}

We denote by $\Omega\defby\{-1,1\}^{\bbZ^2}$ the set of spin configurations. Let $\Lambda$ be a finite subset of $\bbZ^2$, which we denote by $\Lambda\Subset\bbZ^2$. The \textsf{finite-volume Gibbs measure} in $\Lambda$ for the two-dimensional nearest-neighbor ferromagnetic (2d n.n.f.) Ising model, with \textsf{boundary condition} $\omega\in\Omega$ and at \textsf{inverse temperature} $\beta\geq 0$, is the probability measure on $\Omega$ (with the associated product $\sigma$-algebra) defined by
\[
\mu^\omega_{\Lambda;\beta} (\sigma) \defby
\begin{cases}
\frac1{\bfZ^\omega_{\Lambda;\beta}} e^{-\beta H_{\Lambda}(\sigma)}
&
\text{if $\sigma_i=\omega_i$, for all $i\in\Lambda^\comp$,}\\
0 & \text{otherwise,}
\end{cases}
\]
where $\Lambda^\comp\defby\bbZ^2\setminus\Lambda$, and the normalization constant $\bfZ^\omega_{\Lambda;\beta}$ is the partition function. The Hamiltonian in $\Lambda$ is given by
\[
H_{\Lambda}(\sigma) \defby - \sum_{\substack{\{i,j\}\cap\Lambda\neq\varnothing\\\normI{i-j}=1}} \sigma_i\sigma_j.
\]
In particular, we denote by $\mu^+_{\Lambda;\beta}$, resp.\ $\mu^-_{\Lambda;\beta}$, the measures obtained using $\omega\equiv 1$, resp.\ $\omega\equiv -1$.

For $A\subset\bbZ^2$, we denote by $\calF_A$ the $\sigma$-algebra of all events depending only on the spins inside $A$.
A probability measure $\mu$ on $\Omega$ is an \textsf{infinite-volume Gibbs measure} for the 2d n.n.f.\ Ising model at inverse temperature $\beta$ if and only if it satisfies the DLR equation
\begin{equation}\label{DLR}
\mu (\cdot\, | \calF_{\Lambda^\comp})(\omega) = \mu_{\Lambda;\beta}^\omega, \qquad\text{for $\mu$-a.e.\ $\omega$, and all $\Lambda\Subset\bbZ^2$.}
\end{equation}
We denote by $\calG_\beta$ the set of all such measures.

It is easy to prove that the sequences of measures $(\mu_\Lambda^+)_\Lambda$ and $(\mu_\Lambda^-)_\Lambda$ converge weakly along any increasing sequence of finite sets $\Lambda\nearrow\bbZ^2$, the limit being independent of the sequence chosen. We denote by $\mu^+_\beta$ and $\mu^-_\beta$ the corresponding limits; these two measures are called the pure phases, and referred to as the $+$ and $-$ states, and are easily seen to belong to $\calG_\beta$.  In particular, $\calG_\beta\neq\varnothing$, for all $\beta\geq 0$. 

It is a classical result, valid in a much broader context, that the set $\calG_\beta$ is a simplex (see~\cite{Geo88} for a general reference on Gibbs measures). However, determining explicitly this set for a nontrivial model is a very delicate question.

For the 2d n.n.f.\ Ising model, it is not difficult to prove that $\mu_\beta^+$ and $\mu_\beta^-$ are always extremal elements of $\calG_\beta$, and that the latter set contains a unique element if and only if $\mu^+_\beta=\mu^-_\beta$. It can be proved that the latter condition is satisfied if and only if $\beta\leq\betac$ (the difficult part is the behavior at $\betac$), where $0<\betac<\infty$ is the \textsf{inverse critical temperature}. It follows that, in the non-uniqueness regime $\beta>\betac$, $\calG_\beta$ contains at least the two distinct extremal measures $\mu^+_\beta$ and $\mu^-_\beta$.

In 1975, Messager and Miracle-Sole~\cite{MesMir75} proved that all \emph{translation invariant} infinite-volume Gibbs measures of the 2d n.n.f.\ Ising model are convex combinations of $\mu^+_\beta$ and $\mu^-_\beta$; an earlier result on that problem was obtained by Gallavotti and Miracle-Sole for large enough $\beta$~\cite{GalMir72}. (The corresponding claim for general 2d systems at very low temperature was obtained later in~\cite{DobShl85}.)

At this stage, the problem was thus reduced to proving that there are \emph{no} translation non-invariant infinite-volume Gibbs measures in this model. Important progress was made in 1979 by Russo~\cite{Rus79}, who proved that an infinite-volume Gibbs measure for the 2d n.n.f.\ Ising model which is invariant under translations along one direction is necessarily invariant under all translations. Building up on these earlier results, 
Aizenman~\cite{Aiz80} and Higuchi~\cite{Hig81} (see also~\cite{GeoHig00} for a more recent variant) independently established, in the late 1970s, that all infinite-volume Gibbs measures of the 2d n.n.f.\ Ising model are translation invariant, thus providing a complete description of the set $\calG_\beta$.

\medskip
The goal of the present work is to introduce a new approach to this result, with a number of distinctive advantages:
\begin{itemize}
\item We obtain a finite-volume, quantitative analogue (of course, implying the classical claim). Our error estimate is of the correct order.
\item The scheme of our proof seems more natural, and provides a clear picture of the underlying phenomenon.
\item This new approach relying on other properties of the underlying model, it might be extendable to systems for which the classical approach does not apply.
\end{itemize}
Concerning the last point, it is worth pointing out that one of the main ingredients necessary in order to build up a proof along the lines we use here is the availability of a sharp control of interface properties, such as provided by the Ornstein-Zernike theory developed in~\cite{CamIof02,CamIofVel03,CamIofVel08}. In particular, such estimates are available, e.g., for 2d Potts models below the critical temperature, for which even establishing the infinite-volume claim is an open problem. One of the main difficulties in this program, though, is that the geometry of interfaces is much more complicated in systems with more than 2 phases (in the Ising case, interfaces are always lines connecting two points on the boundary). Such an extension, which requires substantial adaptations of several steps in the arguments below, is in progress~\cite{CDIVinprogress}.

There is one drawback in our approach: It does not imply uniqueness at the critical temperature, while this can be extracted from the classical Aizenman-Higuchi result, e.g., using~\cite{BriLeb86}. However, this  should not be surprising, since we expect that it should also apply to models for which the transition is first-order, such as the 2d Potts model with $q\geq 5$ spin states. In that case, one expects $\calG_{\betac}$ to be the simplex with extremal points given by all $q$ low-temperature pure phases \emph{as well as the high temperature phase}.

Note that the absence of translation non-invariant infinite-volume Gibbs measures is specific to the two-dimensional model: In higher dimensions, it was proved by Dobrushin~\cite{Dob72} that such measures exist at sufficiently large values of $\beta$ (however, all translation invariant measures are still convex combinations of $\mu^+_\beta$ and $\mu^-_\beta$ in this case~\cite{Bod06}).
The main difference between the 2d case and its higher-dimensional counterparts is that interfaces in 2d are one-dimensional objects and as such undergo unbounded fluctuations (with diffusive scaling) at any $\beta<\infty$, while horizontal interfaces in higher dimensions are rigid at large enough values of $\beta$. Actually, the existence of a Brownian bridge diffusive limit in 2d has only been established~\cite{GreIof05} for a single interface, resulting from the so-called Dobrushin boundary condition (earlier results restricted to large $\beta$ include~\cite{Gal72} and~\cite{Hig79}). The behavior of the system under a general boundary condition is the main topic of the present work.

\medskip
We set $\Lambda_r \defby \{-\lfloor r\rfloor,\ldots,\lfloor r \rfloor\}^2$. For $\Lambda\Subset\bbZ^2$, we denote by $\bk{\cdot}^\omega_{\Lambda;\beta}$ the expectation under the (finite-volume) measure $\mu^\omega_{\Lambda;\beta}$ and by $\bk{\cdot}^+_{\beta}$, resp. $\bk{\cdot}^-_{\beta}$, the expectation under the (infinite-volume) measure $\mu^+_{\beta}$, resp. $\mu^-_\beta$.

\smallskip
We shall make use of the following notation: If $R_1$, $R_2$ and $R_3$ are three expressions, depending on various parameters ($\beta$, $n$, $\omega$, etc.), and we write $R_1 = R_2 + O_\beta(R_3)$, this means that there exists a constant $C(\beta)<\infty$, depending on $\beta$ only, such that $|R_1-R_2|\leq C(\beta)R_3$. 

\smallskip
Our main result is the following. The proof can be found in Section~\ref{sec_ProofMain}.

\begin{theorem}\label{thm_main}
Let $\beta>\betac$, $\xi<1/2$ and $\omega\in\Omega$. Then, for any $0<\delta<1/2-\xi$, there exists $n_0=n_0(\beta,\xi,\delta)$ such that, for all $n>n_0$, there exists a constant $\alpha^{n,\omega}(\beta)\in[0,1]$ such that, for all $\calF_{\Lambda_{n^\xi}}$-measurable function $f$,
\[
\bk{f}^\omega_{\Lambda_n;\beta}= \alpha^{n,\omega}\bk{f}^+_\beta + (1-\alpha^{n,\omega})\bk{f}^-_\beta+O_\beta\bigl( 
\normsup{f}\, n^{-\delta}\bigr).
\]
\end{theorem}

\noindent It is not difficult to deduce the Aizenman-Higuchi Theorem from Theorem~\ref{thm_main}.

\begin{corollary}\label{cor_AH}
For any $\beta>\betac$, $\calG_\beta=\setof{\alpha\mu^+ + (1-\alpha)\mu^-}{0\leq \alpha\leq 1}$.
\end{corollary}

It is easy to check that the estimate we have on the error term in Theorem~\ref{thm_main} is essentially optimal (and could be made optimal with a little more care in the estimates, replacing the box $\Lambda_{n^a}$ in the proof by a box $\Lambda_{\epsilon n}$ with $\epsilon$ sufficiently small).
\begin{proposition}\label{prop_optimal}
Let $\beta>\betac$. There exist a local function $f$ and a constant $c=c(\beta)>0$ such that, for all $n$ large enough, one can find $\omega\in\Omega$ with
\[
\inf_{\alpha\in[0,1]} \bigl |\bk{f}^\omega_{\Lambda_n;\beta} - \alpha^{n,\omega}\bk{f}^+_\beta - (1-\alpha^{n,\omega})\bk{f}^-_\beta \bigr| \geq c n^{-1/2}.
\]
\end{proposition}

\section{Remarks and open problems}
In this section, we make some comments about Theorem~\ref{thm_main} and list some natural related problems.

\paragraph{\bf General boxes}
Our first comment is that the choice of a square box $\Lambda_n$ in Theorem~\ref{thm_main} does not restrict its generality. Indeed, similarly to what is done in the proof of Corollary~\ref{cor_AH}, given $\Lambda\subset\bbZ^2$, we can consider the largest box $\Lambda_n\subset\Lambda$ and use the Markov property to deduce that the claim of Theorem~\ref{thm_main} remains true for $\Lambda$ (with this value of $n$). This shows that a small region deep inside a box of arbitrary shape, with arbitrary boundary condition, will fall either deeply inside a region of $+$ phase or of $-$ phase, with high probability.

\paragraph{\bf ``Generic'' boundary condition}
As discussed above, the estimate we have on the error term in Theorem~\ref{thm_main} is essentially optimal.
However, it seems very likely that a ``generic'' boundary condition should yield, with high probability, configurations with no crossing interfaces, which should improve the error term to $e^{-O(n)}$. One of the difficulties is to give a precise meaning to the word ``generic'' in this context. One possible choice would be to sample the boundary condition according to some natural probability measure. Unfortunately, very little is known about the Ising model with a strongly inhomogeneous boundary condition. The only work we are aware of that is related to this question is~\cite{vEnNetSch05}, in which the following result is proved: Let the spins of $\omega$ be independent Bernoulli random variables with parameter $1/2$. Then, for almost all $\omega$, the probability of appearance of an interface goes to zero as the system size goes to infinity, provided that $\beta$ be large enough. This shows that, for a generic boundary condition, typical configurations of the low-temperature Ising model do not possess macroscopic interfaces.

A related issue, whose solution would probably be helpful in making progress in the previously mentioned problem, is that of wetting above an inhomogeneous substrate. Consider a 2d n.n.f.\ Ising model at inverse temperature $\beta>\betac$, in a box $\Lambda_n$ with $+$ boundary condition along the vertical and top sides of the box, and $-$ boundary condition along the bottom side. If the interaction $\sigma_i\sigma_j$ between the spins in the bottom row of $\Lambda_n$ and those outside the box is modified to $h\sigma_i\sigma_j$, with $h>0$, then an interface is present along the bottom wall. As long as $h<h_{\rm\scriptscriptstyle w}(\beta)$, for some explicitly known value $0 < h_{\rm\scriptscriptstyle w}(\beta) < 1$, the interface sticks to the bottom wall, its Hausdorff distance to the wall being $O(\log n)$; this is the so-called \textsf{partial wetting} regime. When $h\geq h_{\rm\scriptscriptstyle w}(\beta)$, the interface is repelled away from the bottom wall, and the Hausdorff distance becomes $O(\sqrt{n})$; this is the \textsf{complete wetting} regime. The transition between these two regimes is called the \textsf{wetting transition}. All this is rather well understood, see~\cite{PfVe96} for a review. Understanding the corresponding problem when the homogeneous boundary field $h$ is replaced by site-dependent boundary fields $h_i$ is much more difficult and still mostly open~\cite{DunTop00}.

A final open problem that might be of interest is to understand how robust the Dobrushin boundary conditions are: Start with such a boundary condition, and randomly flip a density $\rho>0$ of spins; does the macroscopic interface survive? What can be said about the critical $\rho$ at which the macroscopic interface disappears?

\section{Proof of the main result}\label{sec_ProofMain}
We shall need several technical results about the 2d n.n.f.\ Ising model. These can be found in Appendix~\ref{appendix}, as well as all relevant definitions for the proofs we present below. We urge the reader not familiar with duality or the random-line representation to read this appendix first.

\medskip
The proof of Theorem~\ref{thm_main} comprises two main steps: (i) Proving that, with high probability, at most one interface approaches the center of the box $\Lambda_n$, (ii) proving that this interface, when present, undergoes unbounded fluctuations (actually of order $\sqrt{n}$). It will then follow that any local observable, with support close to the center of the box, will lie, with high probability, deep inside the $+$ or $-$ phase.

\subsection{Typical configurations have at most one interface near the center of the box}
As explained in Appendix~\ref{appendix}, we associate to the boundary condition $\omega$ the set $\ubo\equiv\{b_1,\ldots,b_{2M}\}$ of endpoints of the open contours induced by $\omega$. We also denote by $\uG(\sigma) \equiv \{\Gamma_1(\sigma),\ldots,\Gamma_{M}(\sigma)\}$ the set of the latter open contours in a configuration $\sigma$ compatible with the boundary condition $\omega$ (their ordering is chosen according to some fixed, but arbitrary, rule). $\uG$ induces a matching of the elements of $\ubo$. Of course, not all possible matchings of $\ubo$ can be realized in this way, and we denote by $\Pi(\omega)$ the set of all admissible matchings; a particular admissible matching, realized in a configuration $\sigma$, is denoted by $\pi(\sigma)$. The notation $(b,b')\in\pi(\sigma)$ means that $b$ and $b'$ are matched in $\pi(\sigma)$. The open contour with endpoints $b$ and $b'$ is denoted by $\Gamma_{b,b'}$.

Let $\max\lbrace 2\xi,\tfrac34\rbrace < a < 1$, and set $\bar\Lambda_{2n^a}\defby [-2n^a,2n^a]^2 \subset \bbR^2$. The next lemma shows that, with high probability, a pair $(b,b')$ in an admissible matching, whose associated open contour intersects the box $\Lambda_{n^a}^\star$, must be such that the segment $\overline{bb'}$ intersects $\bar\Lambda_{2n^a}$.

\begin{lemma}\label{lem_crossing}
Let $\max\lbrace2\xi,\tfrac34\rbrace < a < 1$.
There exists $\Cl{c7}(\beta)>0$ such that, for all $n$ large enough,
\[
\GM\bigl( \exists (b,b')\in\pi(\sigma) \,:\, \Gamma_{b,b'} \cap\Lambda_{n^a}^\star \neq \varnothing,\, \overline{bb'} \cap \bar\Lambda_{2n^a} = \varnothing \bigr) \leq e^{-\Cr{c7} n^{2a-1}}.
\]
\end{lemma}

\begin{proof}
\begin{figure}[t]
 \centering
 \scalebox{.5}{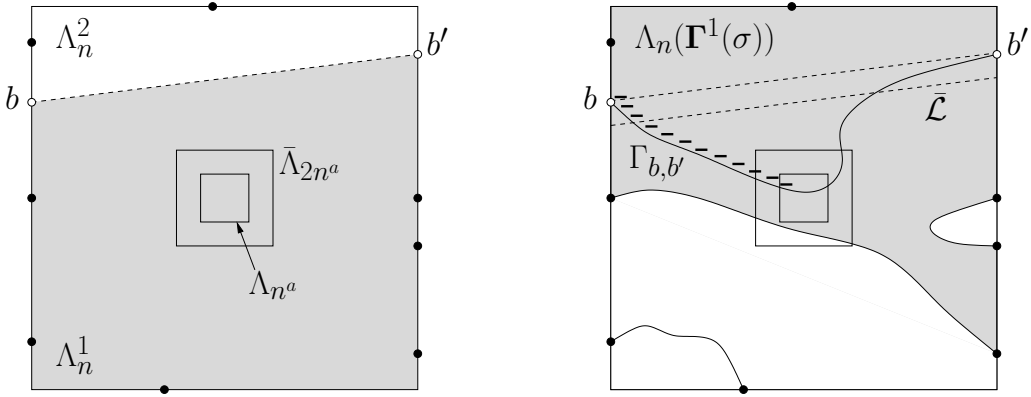}
 \caption{The procedure in Lemma~\ref{lem_crossing}. The dots on the boundary represent $\ubo$, the white ones standing for $b,b'$. Left: The shaded area is the sub-box $\Lambda_n^1$. Right: The shaded area is the box $\Lambda_n(\uG^1)$; observe that when $\Gamma_{b,b'}$ intersects $\Lambda_{n^a}^\star$, there must be an s-path of $-$ spins starting from $\partial\Lambda^2_n\cap\partial\Lambda_n$ and crossing $\bar\calL$ (assuming that the b.c. on $\partial\Lambda_n(\uG^1)\setminus\partial\Lambda_n$ is $+$).}
 \label{fig:crossing}
\end{figure}%
Let $(b,b')\in\ubo$, such that $\overline{bb'} \cap \bar\Lambda_{2n^a} = \varnothing$.
The line segment $\overline{bb'}$ splits $\Lambda_n$ into two disjoint components $\Lambda_n^1$ and $\Lambda_n^2$ (with a fixed rule for attributing the vertices falling on the segment to one of these two sets), with $\bar\Lambda_{2n^a}\subset\Lambda_n^1$ (see Fig.~\ref{fig:crossing}). We denote by $\ub^1(\omega)$ the subset of $\ubo\setminus\{b,b'\}$ consisting of vertices lying on $\partial\Lambda_n^1$.

Let $\calC_{b,b'}$ be the set of configurations of all open contours $\uG^1(\sigma)$ with (both) endpoints in $\ub^1(\omega)$ appearing in configurations $\sigma$ for which $\Gamma_{b,b'}\cap\Lambda_{n^a}^\star\neq\varnothing$.

Such a family $\uG^1(\sigma)$ partitions $\Lambda_n$ into a number of connected components, only one of which contains $b$ and $b'$ along its boundary; we denote the latter component by $\Lambda(\uG^1(\sigma))$, and the corresponding boundary condition by $\omega(\uG^1(\sigma))$ (see Fig.~\ref{fig:crossing}); we assume, without loss of generality, that the boundary condition along $\partial\Lambda(\uG^1(\sigma))\setminus\partial\Lambda_n$ is given by $+$ spins. Using these notations and the DLR equation~\eqref{DLR}, we can write
\[
\GM\bigl(\Gamma_{b,b'} \cap\Lambda_{n^a}^\star \neq \varnothing \bigr) = \sum_{\uG^1 \in \calC_{b,b'}} \GM(\uG^1(\sigma) = \uG^1)\, \mu_{\Lambda(\uG^1);\beta}^{\omega(\uG^1)}\bigl(\Gamma_{b,b'} \cap\Lambda_{n^a}^\star \neq \varnothing \bigr).
\]
Denote by $\bar\calL$ the line parallel to $\overline{bb'}$ at distance $n^a$ from the latter, and located on the same side as $\bar\Lambda_{2n^a}$, and $\calL$ a discrete approximation in $(\bbZ^{2})^\star$ (say, the nearest neighbor path staying closest to $\bar\calL$ in Hausdorff distance, with a fixed rule to break possible ties). On the event $\Gamma_{b,b'} \cap \Lambda_{n^a}^\star\neq\varnothing$, there must be a s-path (see the Appendix) of $-$ spins connecting $\partial\Lambda^2_n\cap\partial\Lambda_n$ to $\calL$, an event we denote by $\partial\Lambda^2_n\cap\partial\Lambda_n \minusconn \calL$. The latter event being decreasing, it follows from the FKG inequality that
\begin{align}\label{fkg}
\mu_{\Lambda(\uG^1);\beta}^{\omega(\uG^1)}\bigl(\Gamma_{b,b'} \cap\Lambda_{n^a}^\star \neq \varnothing \bigr)
&\leq
\mu_{\Lambda(\uG^1);\beta}^{\omega(\uG^1)}\bigl( \partial\Lambda^2_n\cap\partial\Lambda_n \minusconn \calL \bigr) \nonumber \\
&\leq
\mu_{\Lambda_n;\beta}^{\pm(b,b')} \bigl( \partial\Lambda^2_n\cap\partial\Lambda_n \minusconn \calL \bigr) \nonumber \\
&\leq
\mu_{\Lambda_n;\beta}^{\pm(b,b')} \bigl( \Gamma_{b,b'} \cap \calL \neq\varnothing \bigr),
\end{align}
where the boundary condition $\pm(b,b')$ is given by $+1$ along $\partial\Lambda_n^1$ and $-1$ along $\partial\Lambda_n^2$. The last identity follows from the fact that the contour $\Gamma_{b,b'}$ cannot cross an s-path of $+$ spins.

To evaluate the probability in the right-hand side of~\eqref{fkg}, first observe that
\begin{equation}\label{upperbound}
\mu_{\Lambda_n;\beta}^{\pm(b,b')} \bigl( \Gamma_{b,b'} \cap \calL \neq\varnothing \bigr)
\leq
\frac{ \bfZ_{\Lambda_n;\beta}^{+} }{ \bfZ_{\Lambda_n;\beta}^{\pm(b,b')} }\, \sum_{z\in\calL\cap\Lambda_n^\star} \sum_{\Gamma : b\to z \to b'} q_{\Lambda_n;\beta} (\Gamma) .
\end{equation}
On the one hand, applying Lemma~\ref{lem_OZ-rough} with $\rho\in (1/2,2a-1)$, we obtain, for some constant $\Cl{c111}(\beta)$ that
\[
\frac{ \bfZ_{\Lambda_n;\beta}^{\pm(b,b')} }{ \bfZ_{\Lambda_n;\beta}^{+} } \geq e^{-\Cr{c111} n^\rho} e^{-\tau_\beta(b-b')}.
\]
On the other hand, it follows from~\eqref{eq_App_BK_bis} that
\[
\sum_{\Gamma : b\to z \to b'} q_{\Lambda_n;\beta} (\Gamma) \leq e^{-\tau_\beta(z-b) -\tau_\beta(z-b')}.
\]
However, Inequality~\eqref{eq_App_STI} implies that, uniformly in $z\in\calL\cap\Lambda_n^\star$ and in $b,b'$ such that $\overline{bb'}\cap\bar\Lambda_{2n^a}=\varnothing$,
\[
\tau_\beta(z-b)+\tau_\beta(z-b')-\tau_\beta(b'-b) \geq \kappa_\beta \bigl( \normII{z-b} + \normII{z-b'} - \normII{b'-b} \bigr) \geq \Cl{ccc5}(\beta) n^{2a-1}.
\]
Indeed, the triangle $bzb'$ has a base $bb'$ of length less than $3n$ and height at least $n^{a}$.
Since there are at most $4n$ vertices $z\in\calL$, we thus conclude that, for $n$ large enough,
\[
\mu_{\Lambda_n;\beta}^{\pm(b,b')} \bigl( \Gamma_{b,b'} \cap \calL \neq\varnothing \bigr) \leq e^{-\Cl{c3}n^{2a-1}},
\]
for some constant $\Cr{c3}(\beta)>0$. We thus obtain from~\eqref{fkg} that, for all $n$ large enough,
\[
\GM\bigl(\Gamma_{b,b'} \cap\Lambda_{n^a}^\star \neq \varnothing \bigr) \leq e^{-\Cr{c3} n^{2a-1}},
\]
and the conclusion follows, since there are at most $64 n^2$ pairs $b,b'$.
\end{proof}
\begin{lemma}\label{lem_atMostOne}
Let us denote by $N_{\rm\scriptscriptstyle cr}$ the number of open contours intersecting $\Lambda_{n^a}^\star$ (which we call \textsf{crossing contours}). There exists $\Cl{c-atMostOne}(\beta)>0$ such that, for all $n$ large enough,
\[
\GM\bigl( N_{\rm\scriptscriptstyle cr} \geq 2 \bigr) \leq e^{-\Cr{c-atMostOne} n^{2a-1}}.
\]
\end{lemma}
\begin{proof}
\begin{figure}[t]
 \centering
 \scalebox{.5}{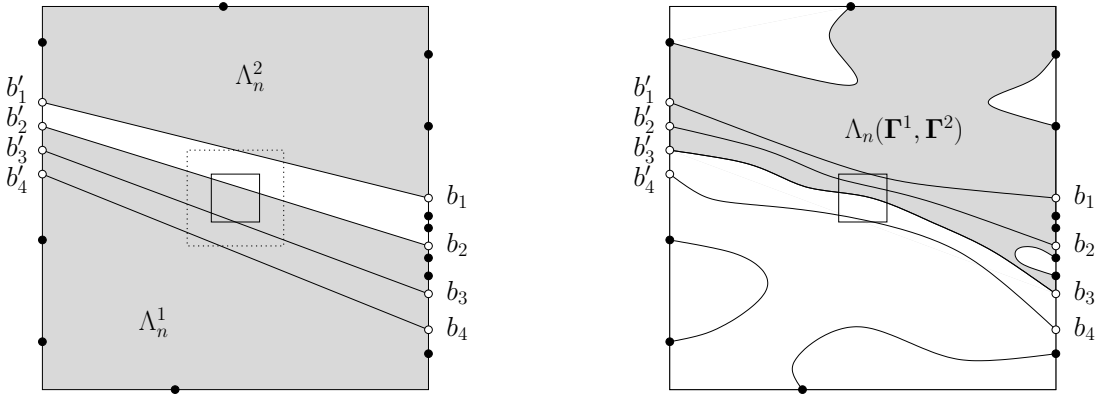}
 \caption{Illustration of the procedure in the proof of Lemma~\ref{lem_atMostOne}.}
 \label{fig:interfaces}
\end{figure}%
Thanks to Lemma~\ref{lem_crossing}, we can assume that all crossing contours have endpoints $b,b'$ satisfying $\overline{bb'}\cap \bar\Lambda_{2n^a}\neq\varnothing$; let us denote by $\calD$ this event.

Let $\Gamma_{b_1,b'_1}(\sigma), \ldots, \Gamma_{b_m,b'_m}(\sigma)$ be the family of all crossing contours in a configuration $\sigma\in\calD$, assuming that $m\geq 2$. Because we suppose that the event $\calD$ is realized, these endpoints can be naturally split into two ``diametrically opposed" families $b_1,\ldots,b_m$ and $b'_1,\ldots,b'_m$. The vertices $b_1,\ldots, b_m$ are ordered clockwise (and thus the corresponding vertices $b'_1,\ldots,b'_m$ counterclockwise). In particular, the crossing contours $\Gamma_{b_1,b'_1}(\sigma)$ and $\Gamma_{b_2,b'_2}(\sigma)$ are neighbors (i.e. there are no other crossing contours between them). Notice that, since $\calD$ is supposed to hold, $\max\lbrace\normI{b_1-b_2},\normI{b'_1-b'_2}\rbrace \leq \Cl{c9} n^{a}$.

The segments $\overline{b_1b'_1}$ and $\overline{b_2b'_2}$ split the box $\Lambda_n$ into 3 pieces. We denote by $\Lambda_n^1$ and $\Lambda_n^2$ the two non-neighboring ones (see Fig.~\ref{fig:interfaces}). Let also $\uG^1$, resp.\ $\uG^2$, be the open contours with both endpoints on $\partial\Lambda_n^1$, resp.\ $\partial\Lambda_n^2$. These open contours partition $\Lambda_n$ into connected pieces, exactly one of which contains $b_1,b_1',b_2,b_2'$ along its boundary; we denote this component by $\Lambda_n(\uG^1,\uG^2)$, and the induced boundary condition on $\Lambda_n(\uG^1,\uG^2)$ by $\omega(\uG^1,\uG^2)$ (see Fig.~\ref{fig:interfaces}). For definiteness and without loss of generality, we can assume that the boundary condition acting along $\partial\Lambda_n(\uG^1,\uG^2)\setminus\partial\Lambda_n$ is given by $+$ spins.
Using the DLR equation (\ref{DLR}), we have
\begin{multline*}
\GM\bigl( N_{\rm\scriptscriptstyle cr} \geq 2,\calD \bigr)
\leq
\sum_{b_1,b'_1,b_2,b'_2} \sum_{\uG^1,\uG^2} \GM(\uG^1(\sigma)=\uG^1,\uG^2(\sigma)=\uG^2)\\
\times \mu_{\Lambda_n(\uG^1,\uG^2);\beta}^{\omega(\uG^1,\uG^2)} \bigl( \Gamma_{b_1,b'_1}\text{ and }\Gamma_{b_2,b'_2} \text{ are crossing} \bigr).
\end{multline*}
Let $\{k_1,\ldots,k_\ell\} = \ub(\omega(\uG^1,\uG^2)) \setminus \{b_1,b_2,b'_1,b'_2\}$ be the set of all endpoints of open contours induced by the boundary condition $\omega(\uG^1,\uG^2)$, apart from $b_1,b_2,b'_1,b'_2$. Using~\eqref{eq_App_marginal}, we obtain
\begin{align*}
\mu_{\Lambda_n(\uG^1,\uG^2);\beta}^{\omega(\uG^1,\uG^2)} \bigl( \Gamma_{b_1,b'_1},\Gamma_{b_2,b'_2} \text{ crossing} \bigr)
&\leq
\frac{\bfZ{_{\Lambda_n(\uG^1,\uG^2);\beta}^+}} {\bfZ{_{\Lambda_n(\uG^1,\uG^2);\beta}^{\omega(\uG^1,\uG^2)}}} \sum_{\substack{\Gamma_1:b_1\to b'_1\\\Gamma_2:b_2\to b'_2}} q_{\Lambda_n(\uG^1,\uG^2);\beta}(\Gamma_1,\Gamma_2).
\end{align*}
On the one hand, using~\eqref{eq_App_BK} and~\eqref{eq_App_upperBd}, we deduce the following upper bound
\begin{align*}
\sum_{\substack{\Gamma_1:b_1\to b'_1\\\Gamma_2:b_2\to b'_2}} q_{\Lambda_n(\uG^1,\uG^2);\beta}(\Gamma_1,\Gamma_2)
& \leq
\sum_{\Gamma_1:b_1\to b'_1}q_{\Lambda_n(\uG^1,\uG^2);\beta}(\Gamma_1)
\sum_{\Gamma_2:b_2\to b'_2}q_{\Lambda_n(\uG^1,\uG^2);\beta}(\Gamma_2)  \\ 
&\leq
 e^{-\tau_\beta(b'_1-b_1)-\tau_\beta(b'_2-b_2)} \leq e^{-\Cl{ccc456}(\beta)n}.
\end{align*}
On the other hand, we evidently have the lower bound
\[
\bfZ_{\Lambda_n(\uG^1,\uG^2);\beta}^{\omega(\uG^1,\uG^2)} \geq e^{-\Cl{c10}n^{a}}\, \bfZ_{\Lambda_n(\uG^1,\uG^2);\beta}^+,
\]
for some constant $\Cr{c10}(\beta)<\infty$, since $\max\lbrace\normI{b_1-b_2},\normI{b'_1-b'_2}\rbrace \leq \Cr{c9} n^{a}$. Combining these estimates, we deduce that
\[
\GM\bigl( N_{\rm\scriptscriptstyle cr} \geq 2 ,\calD \bigr) \leq \Cl{ccc6}n^{2+2a} e^{-\Cl{ccc7} n} \leq  e^{-\Cl{c222} n},
\]
for some constant $\Cr{c222}(\beta)>0$ and for all $n$ large enough.
\end{proof}

\subsection{When present, this interface has large fluctuations}
We denote by $\calI_1$ the event that there is a unique crossing contour.
To deal with $\calI_1$, we have to exploit the fact that the interface undergoes fluctuations of order $\sqrt{n}$ and will thus ``miss'', with high probability, a box of sidelength $n^\xi$ with $\xi<1/2$. The next lemma implements this idea.
\begin{lemma}\label{lem_fluctuations}
Denoting by $\Gamma$ the unique crossing contour on the event $\calI_1$, we have
\[
\GM(\Gamma \cap \Lambda_{2n^\xi}^\star \neq \varnothing, \calI_1) \leq \Cl{c-fluct} n^{\xi-a/2},
\]
for some constant $\Cr{c-fluct}(\beta)$ and all $n$ large enough.
\end{lemma}
\begin{proof}
\begin{figure}[t]
 \centering
 \scalebox{.5}{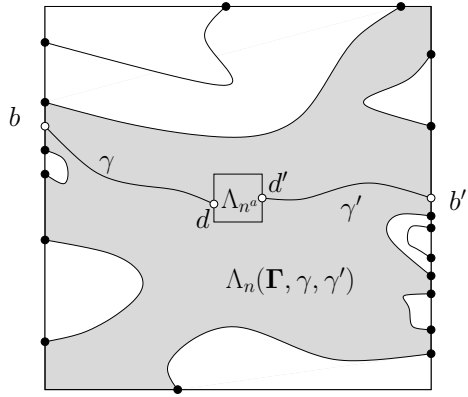}
 \caption{The construction in Lemma~\ref{lem_fluctuations}.}
 \label{fig:fluctuation}
\end{figure}%
Let us denote by $b$ and $b'$ the endpoints of the unique crossing contour $\Gamma$. We denote by $\gamma$ and $\gamma'$ the parts of $\Gamma$ connecting, respectively, $b$ to $\partial^\star\Lambda_{n^a}$ and $b'$ to $\partial^\star\Lambda_{n^a}$ ($\gamma,\gamma'$ are thus two open contours). Let also $\bar\uG$ denote the set of all open contours of the configuration apart from $\Gamma$. The contours $\bar\uG,\gamma,\gamma'$ partition $\Lambda_n$ in a number of connected components, only one of which contains $\Lambda_{n^a}$; we denote the latter by $\Lambda_n(\bar\uG,\gamma,\gamma')$ (see Fig.~\ref{fig:fluctuation}). Let $d,d'$ be the endpoints of $\gamma$ and $\gamma'$ on $\partial^\star\Lambda_{n^a}$.
Observe that the boundary condition acting on $\Lambda_n(\bar\uG,\gamma,\gamma')$ takes two different constant values along each of the two pieces between $d$ and $d'$; we write $\pm(d,d')$ for this boundary condition (by symmetry, it does not matter which part is $+$ and which is $-$).
We consider two cases.\\

\noindent
\emph{Case 1:} $\overline{dd'} \cap \bar\Lambda_{n^a/2} = \varnothing$. In that case, we argue exactly as in the proof of Lemma~\ref{lem_crossing} to obtain that
\[
\mu_{\Lambda_n(\bar\uG,\gamma,\gamma');\beta}^{\pm(d,d')} (\Gamma_{d,d'} \cap \Lambda_{2n^\xi}^\star \neq \varnothing) \leq e^{-\Cl{ccc11}(\beta)n^a}.
\]

\noindent
\emph{Case 2:} $\overline{dd'} \cap \bar\Lambda_{n^a/2} \neq \varnothing$. The argument is completely similar to the one used in the proof of Lemma~\ref{lem_crossing} until expression (\ref{upperbound}). However, the sharp triangle inequality doesn't provide anymore an exponentially small term uniformly over all $\Gamma_{d,d'}$ considered here, since the interface can be straight. We then have to keep track of the prefactors. 
On the one hand, Lemma~\ref{lem_OZ-2ptf} can be applied in order to get
\[
\frac{\bfZ_{\Lambda_n(\bar\uG,\gamma,\gamma');\beta}^{\pm(d,d')}} {\bfZ_{\Lambda_n(\bar\uG,\gamma,\gamma');\beta}^+} \geq \frac{\Cl{ccc8}(\beta)}{n^{a/2}} e^{-\tau_\beta(d'-d)}.
\]
On the other hand, by~\eqref{eq_App_BK_bis}, uniformly in $z\in\partial^\star\Lambda_{2n^\xi}$,
\[
\sum_{\lambda : d\to z \to d'} q_{\Lambda_n(\bar\uG,\gamma,\gamma');\beta} (\lambda) \leq \frac{\Cl{c555}(\beta)}{n^a}\, e^{-\tau_\beta(z-d) -\tau_\beta(z-d')} \leq \frac{\Cr{c555}(\beta)}{n^a}\, e^{-\tau_\beta(d'-d)}.
\]
Summing over $z\in\partial^\star\Lambda_{2n^\xi}$ shows that
\[
\mu_{\Lambda_n(\bar\uG,\gamma,\gamma');\beta}^{\pm(d,d')} (\Gamma_{d,d'} \cap \Lambda_{2n^\xi}^\star \neq \varnothing) \leq \Cl{ccc9}(\beta)\, |\partial^\star\Lambda_{2n^\xi}|\, n^{-a/2} \leq \Cl{ccc10}(\beta)\, n^{\xi-a/2}.
\]
\end{proof}

\subsection{Proof of Theorem \ref{thm_main} and Corollary \ref{cor_AH}.}
\begin{proof}[Proof of Theorem \ref{thm_main}]
\begin{figure}[t]
 \centering
 \scalebox{.5}{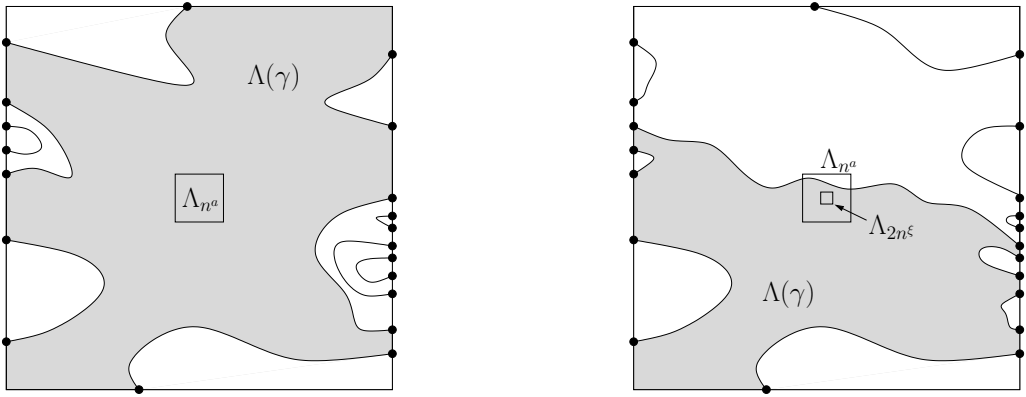}
 \caption{Left: On the event $\calI_0$, there is a region $\Lambda(\gamma)$ (shaded) containing $\Lambda_{n^a}$ with constant spin value on its boundary. Right: On the event $\calI_1$, there is a region $\Lambda(\gamma)$ (shaded) containing $\Lambda_{2n^\xi}$ with constant spin value on its boundary.}
 \label{fig:circuit}
\end{figure}%
Let $\calI_0$ be the event that there is no crossing interface, and, as before, $\calI_1$ the event that there is a unique crossing contour. We know from Lemma~\ref{lem_atMostOne} that, uniformly in $f$,
\begin{equation}\label{eq_I0I1}
 \bk{f}_{\Lambda_n;\beta}^\omega =
\bk{f \,|\, \calI_0}_{\Lambda_n;\beta}^\omega\,\GM(\calI_0) +
\bk{f \,|\, \calI_1}_{\Lambda_n;\beta}^\omega\,\GM(\calI_1) +
O_\beta \bigl(\normsup{f} e^{-\Cr{c-atMostOne}(\beta) n^{2a-1}}\bigr).
\end{equation}
Let us consider first the event $\calI_0$. When the latter occurs, there must be a circuit surrounding $\Lambda_{n^a}$ along which spins take a constant value, see~Fig.~\ref{fig:circuit}. Let us denote by $\calI_0^+(\gamma)$, $\calI_0^-(\gamma)$ the events that the largest such circuit is given by $\gamma$, and the spins value along $\gamma$ is $1$, resp.\ $-1$. Let us also denote by $\Lambda(\gamma)$ the interior of the circuit $\gamma$. It then follows from~\eqref{eq_App_expRelax} that, for some constant $\Cl{cexp}(\beta)>0$, and uniformly in all $\calF_{\Lambda_{n^\xi}}$-measurable functions $f$,
\begin{align}
\bk{f \,|\, \calI_0}_{\Lambda_n;\beta}^\omega
&= \sum_{\gamma} \bigl\{ \GM(\calI_0^+(\gamma) \,|\, \calI_0)\, \bk{f}_{\Lambda(\gamma);\beta}^+ + \GM(\calI_0^-(\gamma) \,|\, \calI_0)\, \bk{f}_{\Lambda(\gamma);\beta}^- \bigr\}\nonumber\\
&= \GM(\calI_0^+\,|\,\calI_0)\, \bk{f}_{\beta}^+ + \GM(\calI_0^-\,|\,\calI_0)\, \bk{f}_{\beta}^- + O_\beta\bigl(\normsup{f} e^{-\Cr{cexp} n^a}\bigr),
\label{eq_I0}
\end{align}
where $\calI_0^\pm \defby \bigcup_\gamma \calI_0^\pm(\gamma)$.\medskip

Now let us consider the event $\calI_1$. It follows from Lemma \ref{lem_fluctuations} that, conditionally on $\calI_1$, there is, with high probability, a contour surrounding $\Lambda_{2n^\xi}$ along which spins take a constant value, see~Fig.~\ref{fig:circuit}. Denoting as before the largest such contour by $\gamma$, its interior by $\Lambda(\gamma)$, and introducing the events $\calI_1^+(\gamma)$ and $\calI_1^-(\gamma)$ similarly as above, we obtain in the same way that, for any $\calF_{\Lambda_{n^\xi}}$-measurable function $f$,
\begin{equation}\label{eq_I1}
\bk{f \,|\, \calI_1}_{\Lambda_n;\beta}^\omega
=
\GM(\calI_1^+\,|\,\calI_1)\, \bk{f}_{\beta}^+ + \GM(\calI_1^-\,|\,\calI_1)\, \bk{f}_{\beta}^- + O_\beta\bigl(\normsup{f} n^{\xi-a/2}\bigr),
\end{equation}
where $\calI_1^\pm \defby \bigcup_\gamma \calI_1^\pm(\gamma)$. \medskip

Let $\calI^\pm \defby \calI_0^\pm\cup\calI_1^\pm$. Observe that $\GM(\calI^+)+\GM(\calI^-) = 1 + O_\beta(n^{\xi-a/2})$.
Inserting~\eqref{eq_I0} and~\eqref{eq_I1} into~\eqref{eq_I0I1}, we obtain finally
\[
\bk{f}_{\Lambda_n;\beta}^\omega
=
\GM(\calI^+)\, \bk{f}_{\beta}^+ + \GM(\calI^-)\, \bk{f}_{\beta}^- + O_\beta\bigl(\normsup{f} n^{\xi-a/2}\bigr),
\]
uniformly in $\calF_{\Lambda_{n^\xi}}$-measurable functions $f$. In particular, we recover the statement of the theorem,
\[
\bk{f}^\omega_{\Lambda_n} = \alpha^{n,\omega}\bk{f}^+ + (1-\alpha^{n,\omega})\bk{f}^- +O_\beta(\normsup{f}\, n^{\xi-a/2}),
\]
by choosing $a=2(b+\delta)$.
\end{proof}

\begin{proof}[Proof of Corollary \ref{cor_AH}]
Let $\mu\in\calG_\beta$ be an infinite-volume Gibbs measure and $f$ be a local function. Let $n_0$ be such that $f$ is $\calF_{\Lambda_{n^\xi}}$-measurable for all $n\geq n_0$.\\
Now from the DLR equation~\eqref{DLR}, we get that, for all $n\geq 1$, and any function $g$,
\[
\mu(g) = \int \bk{g}_{\Lambda_n;\beta}^\omega\, d\mu(\omega).
\]
Theorem~\ref{thm_main} thus implies that, for some $\delta>0$ and uniformly in $\calF_{\Lambda_{n^\xi}}$-measurable functions $g$,
\begin{equation}\label{eq_DLR}
\mu(g)=  A_n\bk{g}^+_\beta + (1-A_n)\bk{g}^-_\beta+O_\beta\bigl( n^{-\delta}\normsup{g} \bigr),
\end{equation}
with $A_n=\int \alpha^{n,\omega}\, d\mu(\omega)$.
Applying this to the function $g=\sigma_0$, we deduce that
\[
\mu(\sigma_0) = ( 2A_n -1 )m^\star_\beta + O_\beta(n^{-\delta}),
\]
where we have introduced the \textsf{spontaneous magnetization} $m^\star_\beta\defby\bk{\sigma_0}^+_\beta$. This shows that
\[
A_n = \frac{m^\star_\beta + \mu(\sigma_0)}{2m^\star_\beta} + O_\beta(n^{-\delta}).
\]
Let us set $\alpha \defby (m^\star_\beta + \mu(\sigma_0)) / 2m^\star_\beta$.
Applying now~\eqref{eq_DLR} to the function $g=f$, we see that, for all $n>n_0$,
\[
 \mu(f) =  \alpha \bk{f}^+_\beta + (1-\alpha)\bk{f}^-_\beta +O_\beta(\normsup{f} n^{-\delta}).
\]
Letting $n$ tend to infinity, we conclude that $\mu(f) = \alpha \bk{f}^+_\beta + (1-\alpha)\bk{f}^-_\beta$. Since this holds for any local function $f$, it follows that $\mu=\alpha \mu^+_\beta + (1-\alpha)\mu^-_\beta$.
\end{proof}

\subsection{Proof of Proposition~\ref{prop_optimal}.}
Let us consider the box $\Lambda_n=\{-n,\ldots,n\}^2$ and the boundary condition $\omega_i=+1$ if and only if $i=(i_1,i_2)$ with $i_2>0$ (Dobrushin boundary condition). We denote the corresponding expectation by $\bk{\cdot}^\pm_{\Lambda;\beta}$.

The trick is to consider a local function $f$ for which the expectation $\bk{f}^+_\beta=\bk{f}^-_\beta=0$, since this trivializes the optimization over $\alpha$. 

Let $f_i(\omega)=\omega_{(0,i)}-\omega_{(0,i-1)}$, and 
\[
F(\omega)=\sum_{i=-\lfloor \Cl{C-opt1}n^{1/2}\rfloor+1}^{\lfloor \Cr{C-opt1}n^{1/2}\rfloor}f_i(\omega)=\omega_{(0, \lfloor \Cr{C-opt1}n^{1/2}\rfloor)}-\omega_{(0,-\lfloor \Cr{C-opt1}n^{1/2}\rfloor)},
\]
with $\Cr{C-opt1}$ a large constant, to be chosen below.
Thanks to translation invariance of $\mu^+$ and $\mu^-$, $\bk{f_i}^+_\beta=\bk{f_i}^-_\beta=0$, for all $i$, and thus $\bk{F}^+_\beta=\bk{F}^-_\beta=0$. Let us denote the only open contour by $\gamma$ and its endpoints $a$ and $b$. Let also
\[
\calS= \setof{(\tfrac12,j)\in\Lambda^\star}{|j|>\lfloor \Cr{C-opt1}n^{1/2}\rfloor},
\]
We then have
\begin{align*}
\bk{F}^{\pm}_{\Lambda,\beta}
&\geq
\bk{F \,|\, \gamma\cap\calS=\varnothing}^{\pm}_{\Lambda,\beta}\,\mu^{\pm}_{\Lambda,\beta}(\gamma\cap\calS=\varnothing)-2\mu^{\pm}_{\Lambda,\beta}(\gamma\cap\calS\neq\varnothing).
\end{align*}
{Now, FKG inequality implies that
\[
\bk{F \,|\, \gamma\cap\calS=\varnothing}^{\pm}_{\Lambda,\beta} \geq 2m^\star_\beta,
\]}
while, using~\eqref{eq_App_STI}, \eqref{eq_App_BK_bis} and Lemma~\ref{lem_OZ-2ptf}, we get
\begin{align*}
\mu^{\pm}_{\Lambda,\beta}(\gamma \cap \calS\neq \varnothing)
 &\leq \Cl{C-opt2} \sum_{z\in\calS} \frac {\sqrt{\vert a-b \vert}} {\sqrt{\vert a-z\vert}\sqrt{\vert z-b\vert}} e^{-(\tau_\beta(a-z)+\tau_\beta(z-b)-\tau_\beta(a-b))} \\
 &\leq \frac{\Cl{}}{\sqrt{n}}\sum_{k\geq \lfloor \Cr{C-opt1}\sqrt{n}\rfloor} e^{-\kappa_\beta k^2/2n}
\leq \Cl{23} e^{-\kappa_\beta \Cr{C-opt1}^2/2}.
\end{align*}
Since $|F|\leq 2$, we deduce from the above, choosing $\Cr{C-opt1}$ large enough, that
\[
\bk{F}^{\pm}_{\Lambda,\beta} \geq 2m^\star_\beta(1 - \Cr{23}e^{-\kappa_\beta\Cr{C-opt1}^2/2}) -2 \Cr{23}e^{-\kappa_\beta\Cr{C-opt1}^2/2}> \Cl{C-opt2} > 0,
\]
Now, $F$ being a sum of $2\lfloor\Cr{C-opt1} n^{1/2}\rfloor$ terms, there exists an index $j_0=j_0(n)$ such that 
\[
\bk{f_{j_0}}^\pm_{\Lambda,\beta}>\frac{\Cr{C-opt2}}{2\lfloor \Cr{C-opt1}n^{1/2}\rfloor}>\Cl{C-opt3}n^{-1/2},
\]
for some constant $\Cr{C-opt3}>0$.

At this point, we have very little control on the location of the support of $f_{j_0}$ inside $\Lambda$.
To remedy this, let $\Delta^{j_0}_n=(0,j_0)+\{-\lfloor n/2\rfloor,\ldots,\lfloor n/2\rfloor\}^2$.
Using DLR equation, we can write (the averaging being over $\omega$)
\[
\bk{f_{j_0}}^\pm_{\Lambda,\beta}= \bk{\,\bk{f_{j_0}}_{\Delta^{j_0}_n,\beta}^\omega\,}_{\Lambda,\beta}^\pm>\Cr{C-opt3}n^{-1/2},
\]
so that there exists an $\tilde\omega=\tilde\omega(n)$ for which
\[
\bk{f_{j_0}}_{\Delta^{j_0}_n,\beta}^{\tilde\omega}>\Cr{C-opt3}n^{-1/2}.
\]
This proves, albeit non-constructively, the existence of a constant $\Cl{C-opt4}>0$ and a sequence of boundary conditions $(\omega_m)_{m\geq 1}$ such that, for all $m$ large enough,
\begin{equation}\label{eq:proofopt}
\inf_{\alpha\in \left[0,1\right]} \vert \bk{f}_{\Delta_{m},\beta}^{\omega_m} -\alpha \bk{f}_\beta^+-(1-\alpha)\bk{f}_\beta^-\vert> \Cr{C-opt4} m^{-1/2},
\end{equation}
where $f(\omega) = \omega_{(0,1)} - \omega_{(0,0)}$ and $\Delta_m=\{-m,\ldots, m\}^2$.
\qed
\begin{remark}
We actually expect that~\eqref{eq:proofopt} is satisfied, for the same function $f$, with $\omega$ given by Dobrushin boundary condition.
\end{remark}

\appendix
\section{Some tools}\label{appendix}
In this appendix, we state, mostly without proof, properties and results that are used in our analysis.
\subsection{Surface tension}
Let $\vec n =(\cos\theta,\sin\theta)\in\bbS^1$. The \textsf{surface tension} $\tau_\beta$ in direction $\vec n$ is defined by
\[
\tau_\beta(\vec n) \defby - \lim_{N\to\infty}\frac {\cos\theta}{(2N+1)}\log \frac{ \bfZ_{\Lambda_N;\beta}^{\omega^{\vec n}} } { \bfZ_{\Lambda_N;\beta}^+ },
\]
where the boundary condition $\omega^{\vec n}$ is defined by $\omega^{\vec n}_i = 1$ if $\inprod{i}{\vec n}\geq 0$, and $\omega^{\vec n}_i = -1$ otherwise.

This limit is known to exist for all values of $\beta$. $\tau_\beta$ is positive for all $\beta>\betac$~\cite{LebPfi81} and is continuous (actually real analytic) as a function of $\vec n$~\cite{CamIofVel03}.

It is useful to extend $\tau_\beta$ to a function on $\bbR^2$ by positive homogeneity, setting $\tau_\beta(x) \defby \tau_\beta(\vec n_x)\normII{x}$, where $\vec n_x\defby x/\normII{x}$.
When $\beta>\betac$, the extended function is a norm on $\bbR^2$. Moreover, it satisfies the following \textsf{sharp triangle inequality}, which follows from a combination
of~\cite[Theorem~2.1]{PfVe99} and~\cite[Theorem~B]{CamIofVel03}: For any $\beta>\betac$, there exists a constant $\kappa_\beta>0$ such that
\begin{equation}\label{eq_App_STI}
\tau_\beta (x) + \tau_\beta (y) - \tau_\beta(x+y) \geq \kappa_\beta \bigl( \normII{x} + \normII{y} - \normII{x+y} \bigr), \quad \forall x,y \in \bbR^2.
\end{equation}

\subsection{Random-line representation}
A subset $A\subset\bbZ^2$ is said to be \textsf{(simply) connected} if $\bigcup_{i\in A} \bigl( i+[-\tfrac12,\tfrac12]^2\bigr)$ is (simply) connected.
Let $\Lambda\Subset\bbZ^2$ be simply connected. Let $\omega\in\{-1,1\}^{\bbZ^2}$ be some boundary condition.
To a configuration $\sigma$ compatible with this boundary condition, we associate the set $E(\sigma)$ of all edges of the dual lattice $(\bbZ^{2})^\star \defby(\tfrac12,\tfrac12)+\bbZ^2$ separating a pair $i,j$ of nearest-neighbor vertices such that $\{i,j\}\cap\Lambda\neq\varnothing$ and $\sigma_i\neq\sigma_j$. The set of edges $E(\sigma)$ can be decomposed into a families of self-avoiding lines by applying the following deformation rules at each vertex of the dual lattice at which more than two edges of $E$ meet:
\begin{center}
 \includegraphics[width=\textwidth]{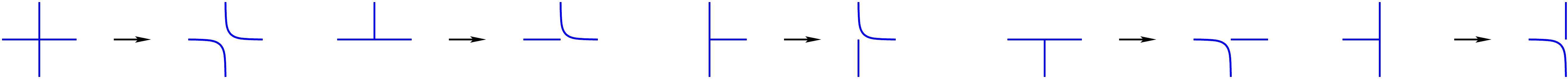}
\end{center}%
Each of these lines is called a \textsf{contour} of $\sigma$. Of particular interest to us are the \textsf{open contours} $\uG(\sigma)=(\Gamma_1(\sigma),\ldots,\Gamma_M(\sigma))$ of the configuration $\sigma$, i.e., the open lines. Observe that each of those has its two endpoints on $\partial^\star\Lambda$, the set of all vertices of $(\bbZ^{2})^\star$ that are at Euclidean distance $1/\sqrt{2}$ from both $\Lambda$ and $\Lambda^\comp$. The set $\ubo\equiv\{b_1,\ldots,b_{2M}\}$ of all endpoints of open contours is completely determined by the boundary condition $\omega$. The notation $\uG\sim\ubo$ means that the set of open contours $\uG$ is compatible with $\ubo$ (i.e., the set of endpoints of $\uG$ is $\ubo$). We also say that a family of open contours $\uG$ is \textsf{$(\omega,\Lambda)$-compatible} if there exists a configuration $\sigma$ in $\Lambda$, compatible with the boundary condition $\omega$, such that $\uG$ is the family of open contours of $\sigma$. We also sometimes use the notation $\Gamma:b\to b'$ in place of $\Gamma\sim\{b,b'\}$.

\smallskip
As a consequence of this particular choice of deformation rules, there is a natural notion of path of vertices of $\bbZ^2$: a sequence $x_1,x_2,\ldots,x_m$ of vertices of $\bbZ^2$ is an \textsf{s-path} if, for all $1\leq i< m$, either $x_i$ and $x_{i+1}$ are nearest-neighbors, or they are second-nearest-neighbors (i.e. at Euclidean distance $\sqrt{2}$ from each other) and oriented NW-SE.

\smallskip
One can define~\cite[(2.10) and Lemma~6.2]{PfVe99} nonnegative weights $q_{\Lambda;\beta}$ on families of open contours in the box $\Lambda$ in such a way that
\begin{equation}\label{eq_App_weight}
\frac{\bfZ_{\Lambda;\beta}^\omega}{\bfZ_{\Lambda;\beta}^+} = \sum_{\uG \sim \ubo} q_{\Lambda;\beta}(\uG) = \bk{\sigma_{b_1}\cdots\sigma_{b_{2M}}}_{\Lambda^\star;\beta^\star},
\end{equation}
where the dual box $\Lambda^\star\defby\setof{t\in(\bbZ^{2})^\star}{d(t,\Lambda)=1/ \sqrt{2}}$, $\beta^\star$ is defined through $\tanh\beta^\star=e^{-2\beta}$, and $\bk{\cdot}_{\Lambda^\star;\beta^\star}$ denotes expectation with respect to the finite-volume Gibbs measure in $\Lambda^\star$ at inverse temperature $\beta^\star$ with free boundary condition,
\[
\mu_{\Lambda^\star;\beta^\star} (\sigma) = \frac 1{\bfZ_{\Lambda^\star;\beta^\star}} \exp\Bigl( -\beta^\star \sum_{\substack{\{i,j\}\subset\Lambda^\star\\\normI{j-i}=1}} \sigma_i\sigma_j \Bigr).
\]
The last identity in~\eqref{eq_App_weight} is a manifestation of the self-duality of the 2d n.n.f.\ Ising model.

The weights $q_{\Lambda;\beta}$ have a number of remarkable properties that make them very useful in the analysis of contours. Here is a list of properties we use in this paper, with precise references to where a proof can be found.
\begin{itemize}
\item Let $i,j\in\partial^\star\Lambda$. Then~\cite[Lemma~6.6 and Prop.~2.4]{PfVe99} 
\begin{equation}\label{eq_App_upperBd}
\sum_{\Gamma:i\to j} q_{\Lambda;\beta} (\Gamma) \leq e^{-\tau_\beta (j-i)}.
\end{equation}
\item We associate to an $(\omega,\Lambda)$-compatible family of open contours the set $\frF(\Gamma_1,\ldots,\Gamma_n)$ of all vertices of $\Lambda$ whose spin value is completely determined by $\omega$ and these open contours, i.e., the maximal set such that, if $\sigma'$ is another configuration compatible with $\omega$ such that $\Gamma_1,\ldots,\Gamma_n \subset \uG(\sigma')$, then $\sigma'_i = \sigma_i$, for all $i\in\frF(\Gamma_1,\ldots,\Gamma_n)$. We set $\Lambda(\Gamma_1,\ldots,\Gamma_n) \defby \Lambda \setminus \frF(\Gamma_1,\ldots,\Gamma_n)$, and say that $\Gamma_1,\ldots,\Gamma_n$ \textsf{partition the box} $\Lambda$ into the connected components of $\Lambda(\Gamma_1,\ldots,\Gamma_n)$. We then have~\cite[Lemma~6.4]{PfVe99}
\begin{equation}\label{eq_App_weightFactorization}
q_{\Lambda;\beta}(\Gamma_1,\ldots,\Gamma_n,\Gamma_{n+1},\ldots,\Gamma_{m}) = q_{\Lambda;\beta}(\Gamma_1,\ldots,\Gamma_n) \, q_{\Lambda(\Gamma_1,\ldots,\Gamma_n);\beta}(\Gamma_{n+1},\ldots,\Gamma_m),
\end{equation}
for all $(\omega,\Lambda)$-compatible family $\Gamma_1,\ldots,\Gamma_m \subset \uG(\sigma)$ of open contours.
\item Let $\ub_1,\ub_2$ be two disjoint subsets of even cardinality of $\partial^\star\Lambda$. The weights satisfy the following BK-type inequality~\cite[Lemma~6.5]{PfVe99},
\begin{equation}\label{eq_App_BK}
\sum_{\substack{\uG_1\sim \ub_1,\uG_2\sim\ub_2\\(\uG_1,\uG_2) \sim \ub_1\cup\ub_2}} q_{\Lambda;\beta} (\uG_1,\uG_2) \leq \sum_{\uG_1\sim \ub_1} q_{\Lambda;\beta} (\uG_1) \sum_{\uG_2\sim \ub_2} q_{\Lambda;\beta} (\uG_2).
\end{equation}
\item Let $z\in\Lambda^\star$; we write $\Gamma:b\to z\to b'$ when $\Gamma:b\to b'$ and $\Gamma\ni z$. Then, as follows from~\cite[Lemma~6.5]{PfVe99} and~\cite[Theorem~A]{CamIofVel03},
\begin{equation}\label{eq_App_BK_bis}
\sum_{\substack{\Gamma:b\to z\to b'}} q_{\Lambda;\beta} (\Gamma) \leq \bk{\sigma_b\sigma_{z}}_{\Lambda^\star;\beta^\star}\bk{\sigma_z\sigma_{b'}}_{\Lambda^\star;\beta^\star} \leq \frac{\Cl{ccc13}(\beta)}{\sqrt{\normII{z-b}\normII{z-b'}}} e^{-\tau_\beta(z-b)-\tau_\beta(z-b')}.
\end{equation}
\item Let $\ub$ be a subset of even cardinality of $\partial^\star\Lambda$, and $b_1,b'_1,b_2,b'_2$ four distinct vertices of $\ub$. Let also $A_1 \subset \bigl\{\Gamma:b_1\to b'_1\bigr\}$ and $A_2 \subset \bigl\{\Gamma:b_2\to b'_2\bigr\}$. It follows easily from~\eqref{eq_App_weight} and~\eqref{eq_App_weightFactorization} that
\begin{equation}\label{eq_App_marginal}
\sum_{\substack{\Gamma_1\in A_1, \Gamma_2\in A_2,\uG\\(\Gamma_1,\Gamma_2,\uG)\sim \ub}} q_{\Lambda;\beta} (\Gamma_1,\Gamma_2,\uG) \leq \sum_{\substack{\Gamma_1\in A_1, \Gamma_2\in A_2\\(\Gamma_1,\Gamma_2)\sim\{b_1,b'_1,b_2,b'_2\}}} q_{\Lambda;\beta} (\Gamma_1,\Gamma_2) .
\end{equation}
\end{itemize}

\subsection{Spatial relaxation in pure phases}
Another result that plays an important role in our analysis is the following exponential relaxation result: Let $\Lambda\subset\bbZ^2$. Then~\cite{BriLebPfi81,ChaChaSch87}, for any $\beta>\betac$, there exists $\Cl{c_relax}(\beta)>0$ such that, uniformly for any local function $f$ with support $S(f)$ inside $\Lambda$,
\begin{equation}\label{eq_App_expRelax}
\bigl| \bk{f}_{\Lambda;\beta}^+ - \bk{f}_{\beta}^+ \bigr| \leq \normsup{f}\,|S(f)|\, e^{-\Cr{c_relax} d(S(f),\Lambda^\comp)}.
\end{equation}
Notice that even though the authors of~\cite{BriLebPfi81,ChaChaSch87} rely on the exact solution to guarantee, respectively, exponential decay of the truncated 2-point function for $\beta>\betac$ or exponential decay of the 2-point function for $\beta<\betac$, one can instead, in both cases, use the positivity of surface tension for $\beta>\betac$ proved in~\cite{LebPfi81} (in the first case, by proving that the truncated 2-point function $\bk{\sigma_0;\sigma_x}^+_\beta$ is bounded above by the probability that $0$ and $x$ are surrounded by a contour; in the second case, by using the fact that the rate of exponential decay of $\bk{\sigma_0\sigma_x}_\beta$ is equal to the surface tension $\tau_\beta(\vec n_x)$, by duality).

\subsection{Finite-volume corrections to $\tau_\beta$}
The next two lemmas provide informations on the finite-volume corrections to the surface tension $\tau_\beta$ and play a crucial role in our analysis.

The first lemma provides a lower bound for the ratio of partition functions in a square box, when the endpoints are not both simultaneously close to one side of the box (in which case, the prefactor would change). With slightly more work, this lower bound can be replaced by full Ornstein-Zernike asymptotics, using a variant of~\cite{CamIofVel03} similarly to what is done in~\cite{GreIof05}.
\begin{lemma}\label{lem_OZ-2ptf}
Let $\beta>\betac$. Then there exists a constant $\Cl{c-OZlb}>0$ such that, uniformly as $n\to\infty$,
\[
\frac{\bfZ_{\Lambda_n;\beta}^{\pm(i,j)}} {\bfZ_{\Lambda_n;\beta}^+} \geq \Cr{c-OZlb}n^{-1/2}\, e^{-\tau_\beta (j-i)},
\]
uniformly in vertices $i,j\in\partial^\star\Lambda_n$ such that the segment $\overline{ij}$ intersects the box $[-n/2,n/2]^2$.
\end{lemma}
\begin{proof}
Using~\cite[Lemma~6.3]{PfVe99}, we can replace the weights $q_{\Lambda_n;\beta}$ by their infinite-volume counterparts $q_\beta$,
\[
\frac{\bfZ_{\Lambda_n;\beta}^{\pm(i,j)}} {\bfZ_{\Lambda_n;\beta}^+} = \sum_{\substack{\gamma:i\to j\\\gamma\subset\Lambda_n}} q_{\Lambda_n;\beta}(\gamma) \geq \sum_{\substack{\gamma:i\to j\\\gamma\subset\Lambda_n}} q_{\beta}(\gamma),
\]
where the condition $\gamma\subset\Lambda_n$ means that all edges of $\gamma$ must have their endpoints in $\Lambda_n^\star$.
We consider two cases.\\
\emph{Case 1:} the angle between the segment $\overline{ij}$ and each diagonal of the square $\Lambda_n$ is greater than $\pi/5$.

In that case, $i$ and $j$ must be on opposite sides of $\Lambda_n$ and at a distance at least $n/4$ from the two other sides. For definiteness, let us assume that $i,j$ are on the two vertical sides of the box. Let $\Delta_n = \{-n,\ldots,n\}\times\bbZ$ be the vertical strip of width $2n+1$ centered at $0$. It follows from~\cite[Lemma~6.10]{PfVe99} that
\[
\sum_{\substack{\gamma:i\to j\\\gamma\subset\Lambda_n}} q_{\beta}(\gamma) \geq (1-o(1))\, \sum_{\substack{\gamma:i\to j\\\gamma\subset\Delta_n}} q_{\beta}(\gamma).
\]
Now, the required bound follows from the Ornstein-Zernike asymptotics derived in~\cite{GreIof05}.\\
\emph{Case 2:} the angle between the segment $\overline{ij}$ and one of the diagonals of the square $\Lambda_n$ is smaller than $\pi/5$.

In this case, one can easily adapt the proof of Ornstein-Zernike asymptotics given in~\cite{CamIofVel03}: Taking a forward-cone (see the latter paper for definition) of sufficiently small opening to ensure that it is contained in the cone $\setof{x=(x_1,x_2)\in\bbR^2}{x_1\geq 0, x_2\geq 0}$, we see that the constraint that $\gamma$ be $\Lambda_n$-compatible only affects the left-most and right-most irreducible pieces. This has no impact on the derivation in~\cite{CamIofVel03}.
\end{proof}
When the endpoints $i$ and $j$ both lie too close to one of the sides of $\Lambda_n$, the above result does not apply (and is actually incorrect in general). It turns out that, for our purposes in this paper, the following rough lower bound is sufficient.
\begin{lemma}\label{lem_OZ-rough}
Let $\beta>\betac$. Then, for any $1/2<\rho<1$, there exists a constant $\Cl{c_rough}=\Cr{c_rough}(\beta)$ such that, for all $i,j\in\partial^\star\Lambda_n$,
\[
\frac{\bfZ_{\Lambda_n;\beta}^{\pm(i,j)}} {\bfZ_{\Lambda_n;\beta}^+} \geq e^{-\Cr{c_rough} n^{\rho}}\, e^{-\tau_{\beta} (j-i)}.
\]
\end{lemma}
\begin{proof}
First, by~\eqref{eq_App_weight},
\[
\frac{\bfZ_{\Lambda_n;\beta}^{\pm(i,j)}} {\bfZ_{\Lambda_n;\beta}^+} = \bk{\sigma_i \sigma_j}_{\Lambda_n^\star;\beta^\star}.
\]
Let $i',j'\in\Lambda_{n-n^{\rho}}$ be the two vertices closest to $i$ and $j$. Then, by the GKS inequality,
\[
\bk{\sigma_i \sigma_j}_{\Lambda_n^\star;\beta^\star} \geq \bk{\sigma_i \sigma_{i'}}_{\Lambda_n^\star;\beta^\star} \bk{\sigma_{i'} \sigma_{j'}}_{\Lambda_n^\star;\beta^\star} \bk{\sigma_{j'} \sigma_{j}}_{\Lambda_n^\star;\beta^\star}.
\]
On the one hand, it follows from the GKS inequality that 
\[\bk{\sigma_i \sigma_{i'}}_{\Lambda_n^\star;\beta^\star}\bk{\sigma_{j'} \sigma_{j}}_{\Lambda_n^\star;\beta^\star} \geq e^{-\Cl{ccc1}(\beta) n^{\rho}}.
\]
On the other hand, it follows from~\cite[Lemma~6.10]{PfVe99} and our choice of $\rho$ that 
\[\bk{\sigma_{i'} \sigma_{j'}}_{\Lambda_n^\star;\beta^\star} = (1+o(1))\, \bk{\sigma_{i'} \sigma_{j'}}_{\beta^\star}\geq \Cl{ccc555} \normII{j'-i'}^{-1/2}\, e^{-\tau_\beta(j'-i')},
\]
 since the infinite-volume 2-point function admits Ornstein-Zernike asymptotics~\cite[Theorem~A]{CamIofVel03}. It then follows from the continuity of $\tau_\beta$ as a function of the direction that
\[
\bk{\sigma_{i'} \sigma_{j'}}_{\Lambda_n^\star;\beta^\star} \geq e^{-\Cl{ccc2}(\beta) n^{\rho}}\, e^{-\tau_{\beta} (j-i)}.
\]
\end{proof}

\paragraph{\textbf{Acknowledgments.}}
We would like to thank Hugo Duminil-Copin, Dima Ioffe and Charles Pfister for comments and encouragements. Support from the Swiss National Science Foundation is also gratefully acknowledged.

\bibliographystyle{siam}
\bibliography{AH-Ising}
 
\end{document}